\documentclass[reqno]{amsart}
\usepackage{amssymb,setspace,geometry}
\usepackage{ifpdf}
\ifpdf
 \usepackage[hyperindex,pagebackref]{hyperref}%
\else
 \expandafter\ifx\csname dvipdfm\endcsname\relax
 \usepackage[hypertex,hyperindex,pagebackref]{hyperref}%
 \else
 \usepackage[dvipdfm,hyperindex,pagebackref]{hyperref}%
 \fi
\fi
\theoremstyle{plain}
\newtheorem{thm}{Theorem}[section]

\newtheorem{cor}{Corollary}[section]
\theoremstyle{remark}
\newtheorem{rem}{Remark}[section]
\allowdisplaybreaks[4]
\DeclareMathOperator{\td}{d}
\numberwithin{equation}{section}

\begin{document}

\title[Integral representation and properties of Cauchy numbers]
{An integral representation, complete monotonicity, and inequalities of Cauchy numbers of the second kind}

\author[F. Qi]{Feng Qi}
\address{College of Mathematics, Inner Mongolia University for Nationalities, Tongliao City, Inner Mongolia Autonomous Region, 028043, China; Department of Mathematics, College of Science, Tianjin Polytechnic University, Tianjin City, 300387, China; Institute of Mathematics, Henan Polytechnic University, Jiaozuo City, Henan Province, 454010, China}
\email{\href{mailto: F. Qi <qifeng618@gmail.com>}{qifeng618@gmail.com}, \href{mailto: F. Qi <qifeng618@hotmail.com>}{qifeng618@hotmail.com}, \href{mailto: F. Qi <qifeng618@qq.com>}{qifeng618@qq.com}}
\urladdr{\url{http://qifeng618.wordpress.com}}

\begin{abstract}
In the paper, the author establishes an integral representation for Cauchy numbers of the second kind, finds the complete monotonicity, minimality, and logarithmic convexity of Cauchy numbers of the second kind, and presents some inequalities for determinants and products of Cauchy numbers of the second kind.
\end{abstract}




\maketitle

\section{Introduction}

According to~\cite[p.~293--294]{Comtet-Combinatorics-74}, there are two kinds of Cauchy numbers which may be defined respectively by
\begin{equation}
C_n=\int_0^1\langle x\rangle_n\td x
\quad\text{and}\quad
c_n=\int_0^1(x)_n\td x,
\end{equation}
where
\begin{equation}
\langle x\rangle_n=
\begin{cases}
x(x-1)(x-2)\dotsm(x-n+1), & n\ge1\\
1,& n=0
\end{cases}
\end{equation}
and
\begin{equation}
(x)_n=
\begin{cases}
x(x+1)(x+2)\dotsm(x+n-1), & n\ge1\\
1, & n=0
\end{cases}
\end{equation}
are respectively called the falling and rising factorials. The coefficients expressing rising factorials $(x)_n$ in terms of falling factorials $\langle x\rangle_n$ are called Lah numbers. Lah numbers have an interesting meaning in combinatorics: they count the number of ways a set of $n$ elements can be partitioned into $k$ nonempty linearly ordered subsets.
Shortly speaking, Cauchy numbers play important roles in some fields, such as approximate integrals, Laplace summation formula, and difference-differential equations, and are also related to some famous numbers such as Stirling numbers, Bernoulli numbers, and harmonic numbers. Therefore, Cauchy numbers deserve to be studied.
\par
It is known~\cite[p.~294]{Comtet-Combinatorics-74} that Cauchy numbers of the second kind $c_k$ may be generated by
\begin{equation}
\frac{-t}{(1-t)\ln(1-t)}=\sum_{n=0}^\infty c_n\frac{t^n}{n!}
\end{equation}
which is equivalent to
\begin{equation}\label{Norlund-No-dfn}
\frac{t}{(1+t)\ln(1+t)}=\sum_{n=0}^\infty(-1)^nc_n\frac{t^n}{n!}.
\end{equation}
The first few Cauchy numbers of the second kind $c_k$ are
\begin{equation}
c_0=1,\quad c_1=\frac12,\quad c_2=\frac56,\quad c_3=\frac94, \quad c_4=\frac{251}{30}, \quad c_5=\frac{475}{12}, \quad c_6=\frac{19087}{84}.
\end{equation}
\par
In this paper, we will establish an integral representation, find the complete monotonicity, minimality, and logarithmic convexity, and present some inequalities of Cauchy numbers of the second kind $c_n$.

\section{An integral representation of Cauchy numbers}

We first establish an integral representation of Cauchy numbers of the second kind $c_n$.

\begin{thm}\label{Norlund-No-int-thm}
For $n\in\{0\}\cup\mathbb{N}$, Cauchy numbers of the second kind $c_n$ have an integral representation
\begin{equation}\label{Norlund-No-int-eq}
c_n=n!\int_0^\infty \frac{\td u}{u[\pi^2+(\ln u)^2](1+u)^{n}}.
\end{equation}
\end{thm}

\begin{proof}
Recall from~\cite{Zhang-Li-Qi-Log.tex} that the function
\begin{equation}\label{F(z)}
F(z)=
\begin{cases}
 \dfrac{z}{(1+z)\ln(1+z)}, & z\in\mathbb{C}\setminus(-\infty,-1]\setminus\{0\}\\
 1, & z=0
\end{cases}
\end{equation}
has the integral representation
\begin{equation}\label{thm-zhang-li-qi-eq}
F(z)=\int_0^\infty \frac{u+1}{u[(\ln u)^2 +\pi^2]}\frac{\td u}{u+1+z},\quad z\in\mathbb{C}\setminus(-\infty,-1],
\end{equation}
where $\mathbb{C}$ is the set of all complex numbers. Differentiating $n$ times on both sides of~\eqref{Norlund-No-dfn} and~\eqref{thm-zhang-li-qi-eq} yields
\begin{equation*}
F^{(n)}(t)=\sum_{k=n}^\infty (-1)^{k}c_k\frac{t^{k-n}}{(k-n)!}
=\sum_{k=0}^\infty (-1)^{k+n}c_{k+n}\frac{t^{k}}{k!}
\end{equation*}
and
\begin{equation*}
F^{(n)}(t)=(-1)^nn!\int_0^\infty \frac{u+1}{u[(\ln u)^2 +\pi^2]}\frac{\td u}{(u+1+t)^{n+1}}.
\end{equation*}
Hence,
\begin{equation*}
\sum_{k=0}^\infty (-1)^{k+n}c_{k+n}\frac{t^{k}}{k!}
=(-1)^nn!\int_0^\infty \frac{u+1}{u[(\ln u)^2 +\pi^2]}\frac{\td u}{(u+1+t)^{n+1}}.
\end{equation*}
Further letting $t\to0$ on both sides of the above equation gives the integral representation~\eqref{Norlund-No-int-eq}. The proof of Theorem~\ref{Norlund-No-int-thm} is complete.
\end{proof}

\section{Complete monotonicity and minimality of Cauchy numbers}

Basing on the integral representation~\eqref{Norlund-No-int-eq}, we now find complete monotonicity and minimality of Cauchy numbers of the second kind $c_n$.

Recall from monographs~\cite[pp.~372--373]{mpf-1993} and~\cite[p.~108, Definition~4]{widder} that a sequence $\{\mu_n\}_{0\le n\le\infty}$ is said to be completely monotonic if its elements are non-negative and its successive differences are alternatively non-negative, that is
\begin{equation}
(-1)^k\Delta^k\mu_n\ge0
\end{equation}
for $n,k\ge0$, where
\begin{equation}
\Delta^k\mu_n=\sum_{m=0}^k(-1)^m\binom{k}{m}\mu_{n+k-m}.
\end{equation}
Recall from~\cite[p.~163, Definition~14a]{widder} that a completely monotonic sequence $\{a_n\}_{n\ge0}$ is minimal if it ceases to be completely monotonic when $a_0$ is decreased.

\begin{thm}\label{Ber-minimal-thm}
The infinite sequence of Cauchy numbers of the second kind
\begin{equation}\label{Cauchy-minimal-seq}
\Bigl\{\frac{c_n}{n!}\Bigr\}_{n\ge0}
\end{equation}
is completely monotonic and minimal.
\end{thm}

\begin{proof}
It was stated in~\cite[pp.~372--373]{mpf-1993} and~\cite[p.~108, Theorem~4a]{widder} that a necessary and sufficient condition that the sequence $\{\mu_n\}_0^\infty$ should have the expression
\begin{equation}\label{mu=alpha-moment}
\mu_n=\int_0^1t^n\td\alpha(t)
\end{equation}
for $n\ge0$, where $\alpha(t)$ is non-decreasing and bounded for $0\le t\le1$, is that it should be completely monotonic. Theorem~14a in~\cite[p.~164]{widder} states that a completely monotonic sequence $\{\mu_n\}_{n\ge0}$ is minimal if and only if the integral representation~\eqref{mu=alpha-moment} is valid for $n\ge0$ and $\alpha(t)$ is a non-decreasing bounded function continuous at $t=0$.
\par
From~\eqref{Norlund-No-int-eq}, it follows that for $n\in\mathbb{N}$
\begin{align*}
\frac{c_n}{n!}&=\int_0^\infty \frac{1}{u[(\ln u)^2 +\pi^2]}\frac{\td u}{(u+1)^{n}}\\
&=\int_1^0 \frac{1}{(-\ln t)\{[\ln(-\ln t)]^2 +\pi^2\}}\frac{\td(-\ln t)}{(1-\ln t)^{n}}\\
&=\int_0^1 \frac{\td t}{t(-\ln t)\{[\ln(-\ln t)]^2 +\pi^2\}(1-\ln t)^{n}}\\
&=\int_0^1 t^n\frac{\td t}{t^{n+1}(-\ln t)\{[\ln(-\ln t)]^2 +\pi^2\}(1-\ln t)^{n}}\\
&=\int_0^1 t^n\td\biggr[\int_0^t\frac1{u^{n+1}(-\ln u) \{[\ln(-\ln u)]^2 +\pi^2\} (1-\ln u)^{n}}\td u\biggr].
\end{align*}
This implies the complete monotonicity and minimality of the sequence~\eqref{Cauchy-minimal-seq}.
\par
The complete monotonicity may be alternatively proved as follows.
In~\cite[p.~373]{mpf-1993}, it was stated that if a function $f(t)$ is completely monotonic on $[0,\infty)$, that is, $(-1)^kf^{(k)}(t)\ge0$ for $k\ge0$, then the sequence $\bigl\{(-1)^nf^{(n)}(n)\bigr\}$ is completely monotonic. It is clear that the function of $x$
\begin{equation*}
\int_0^\infty \frac{1}{u[(\ln u)^2 +\pi^2]}\frac{\td u}{(u+1)^{x}}
\end{equation*}
is completely monotonic on $[0,\infty)$. Hence, by the integral representation~\eqref{Norlund-No-int-eq}, the sequence~\eqref{Cauchy-minimal-seq} is completely monotonic.
The proof of Theorem~\ref{Ber-minimal-thm} is complete.
\end{proof}

\section{Positivity of determinants for Cauchy numbers}
With the help of the integral representation~\eqref{Norlund-No-int-eq}, we now present the positivity of two determinants of Cauchy numbers of the second kind $c_n$.

\begin{thm}\label{b(n)-matrix-thm}
Let $m\in\mathbb{N}$ and let $n$ and $a_k$ for $1\le k\le m$ be nonnegative integers. Then
\begin{equation}\label{matrix-1f}
|(-1)^{a_i+a_j}c_{n+a_i+a_j}|_m\ge0
\end{equation}
and
\begin{equation}\label{matrix-2f}
(-1)^{mn}|c_{n+a_i+a_j}|_m\ge0,
\end{equation}
where $|a_{kj}|_m$ denotes a determinant of order $m$ with elements $a_{kj}$.
\end{thm}

\begin{proof}
From the proof of Theorem~\ref{Norlund-No-int-thm}, we observe that
\begin{equation}\label{b(n)-limit}
\frac{c_n}{n!}=\lim_{t\to0^+}h_n(t),
\end{equation}
where
\begin{equation}\label{h(x)-dfn-eq}
h_n(t)=\int_0^\infty \frac{u+1}{u[(\ln u)^2 +\pi^2]}\frac{\td u}{(u+1+t)^{n+1}}
\end{equation}
is completely monotonic on $[0,\infty)$ and
\begin{equation}\label{c(n)h(n)-limit}
h_n^{(k)}(t)=(-1)^k\frac{(n+k)!}{n!}h_{n+k}(t)
\to(-1)^k\frac{(n+k)!}{n!}\frac{c_{n+k}}{(n+k)!}
=(-1)^k\frac{c_{n+k}}{n!}
\end{equation}
as $t\to0^+$.
\par
In~\cite{two-place}, or see~\cite[p.~367]{mpf-1993}, it was obtained that if $f$ is a completely monotonic function on $[0,\infty)$, then
\begin{equation}\label{matrix-eq-1}
\bigl|f^{(a_i+a_j)}(x)\bigr|_m\ge0
\end{equation}
and
\begin{equation}\label{matrix-eq-2}
\bigl|(-1)^{a_i+a_j}f^{(a_i+a_j)}(x)\bigr|_m\ge0.
\end{equation}
Applying $f$ in~\eqref{matrix-eq-1} and~\eqref{matrix-eq-2} to the function $h_n(x)$ yields
\begin{equation}\label{matrix-eq-1h}
\bigl|h_n^{(a_i+a_j)}(x)\bigr|_m\ge0
\end{equation}
and
\begin{equation}\label{matrix-eq-2h}
\bigl|(-1)^{a_i+a_j}h_n^{(a_i+a_j)}(x)\bigr|_m\ge0.
\end{equation}
Letting $x\to0^+$ in~\eqref{matrix-eq-1h} and~\eqref{matrix-eq-2h} and making use of~\eqref{c(n)h(n)-limit} produce
\begin{equation}\label{matrix-1b}
\biggl|(-1)^{a_i+a_j}\frac{c_{n+a_i+a_j}}{n!}\biggr|_m\ge0
\end{equation}
and
\begin{equation}\label{matrix-2b}
\biggl|(-1)^{n}\frac{c_{n+a_i+a_j}}{n!}\biggr|_m\ge0.
\end{equation}
Further simplifying~\eqref{matrix-1b} and~\eqref{matrix-2b} leads to~\eqref{matrix-1f} and~\eqref{matrix-2f}. The proof of Theorem~\ref{b(n)-matrix-thm} is complete.
\end{proof}

\section{Inequalities for products of Cauchy numbers}
In this final section, by virtue of the integral representation~\eqref{Norlund-No-int-eq}, we discover some inequalities and, as a consequence, the logarithmic convexity of Cauchy numbers of the second kind $c_n$.
\par
Let $\lambda=(\lambda_1,\lambda_2,\dotsc,\lambda_{n})\in\mathbb{R}^{n}$ and $\mu=(\mu_1,\mu_2,\dotsc, \mu_{n})\in\mathbb{R}^{n}$. The sequence $\lambda$ is said to be majorized by $\mu$ \textup{(}in symbols $\lambda\preceq \mu$\textup{)} if
\begin{equation*}
\sum_{\ell=1}^k \lambda_{[\ell]}\le\sum_{\ell=1}^k \mu_{[\ell]}
\end{equation*}
for $k=1,2,\dotsc,n-1$ and
\begin{equation*}
\sum_{\ell=1}^n \lambda_\ell=\sum_{\ell=1}^n\mu_\ell,
\end{equation*}
where $\lambda_{[1]}\ge \lambda_{[2]}\ge \dotsm \ge \lambda_{[n]}$ and $\mu_{[1]}\ge \mu_{[2]}\ge\dotsm \ge \mu_{[n]}$ are rearrangements of $\lambda$ and $\mu$ in a descending order.
A sequence $\lambda$ is said to be strictly majorized by $\mu$ $($in symbols $\lambda \prec \mu)$ if $\lambda$ is not a permutation of $\mu$.

\begin{thm}\label{lambda-mu-thm}
Let $m\in\mathbb{N}$ and let $\lambda$ and $\mu$ be two $m$-tuples of nonnegative integers such that $\lambda\preceq\mu$. Then
\begin{equation}\label{lambda-mu-eq}
\prod_{i=1}^m c_{\lambda_i} \le \prod_{i=1}^m c_{\mu_i}.
\end{equation}
\end{thm}

\begin{proof}
In~\cite[p.~106, Theorem~A]{haerc-JMAA-1997} and~\cite[p.~367, Theorem~2]{mpf-1993}, a minor correction of~\cite[Theorem~1]{finkjmaa82}, it was obtained that if $f$ is a completely monotonic function on $(0,\infty)$ and $\lambda\preceq\mu$, then
\begin{equation}\label{finkjmaa82=ineq3.2}
\Biggl|\prod_{i=1}^nf^{(\lambda_i)}(x)\Biggr|\le \Biggl|\prod_{i=1}^nf^{(\mu_i)}(x)\Biggr|.
\end{equation}
The equality in~\eqref{finkjmaa82=ineq3.2} is valid only when $\lambda$ and $\mu$ are identical or when $f(x)=e^{-cx}$ for $c\ge0$.
Applying the inequality~\eqref{finkjmaa82=ineq3.2} to $h_n(x)$ creates
\begin{equation*}
\Biggl|\prod_{i=1}^mh_{n}^{(\lambda_i)}(t)\Biggr|
\le \Biggl|\prod_{i=1}^m h_{n}^{(\mu_i)}(t)\Biggr|.
\end{equation*}
Taking the limit $t\to0^+$ on both sides of the above inequality and making use of~\eqref{c(n)h(n)-limit} reveal
\begin{equation}\label{Apply-h(n-t)}
\Biggl|\prod_{i=1}^m(-1)^{\lambda_i}\frac{c_{n+\lambda_i}}{n!}\Biggr|
\le \Biggl|\prod_{i=1}^m (-1)^{\mu_i}\frac{c_{n+\mu_i}}{n!}\Biggr|
\end{equation}
which is equivalent to~\eqref{lambda-mu-eq}. The proof of Theorem~\ref{lambda-mu-thm} is complete.
\end{proof}

\begin{cor}\label{b(n)-log-Convex-thm}
The infinite sequence $\{c_{n}\}_{n\ge0}$ is logarithmically convex.
\end{cor}

\begin{proof}
This follows from the majorization relation $(i+2,i)\succeq(i+1,i+1)$ for $i\ge0$ and Theorem~\ref{lambda-mu-thm}.
\par
This may also be verified as follows. In~\cite[p.~369]{mpf-1993} and~\cite[p.~429, Remark]{Remarks-on-Fink}, it was stated that if $f(t)$ is a completely monotonic function such that $f^{(k)}(t)\ne0$ for $k\ge0$, then the sequence
\begin{equation}
s_i(t)=\ln\bigl[(-1)^{i-1}f^{(i-1)}(t)\bigr], \quad i\ge1
\end{equation}
is convex. Applying this result to the function $h_n(t)$ and making use of~\eqref{c(n)h(n)-limit} figures out that the sequence
\begin{equation}
s_i(t)=\ln\bigl[(-1)^{i-1}h_n^{(i-1)}(t)\bigr]
\to \ln\frac{c_{n+i-1}}{n!}, \quad t\to0^+
\end{equation}
for $i\ge1$ is convex. Hence, the sequence $\{c_{n}\}_{n\ge0}$ is logarithmically convex.
\end{proof}

\begin{cor}
For $\ell\ge0$ and $n>k>0$, we have
\begin{equation}\label{cor-2nd-cauchy}
({c_{\ell+k}})^n\le({c_{\ell+n}})^k({c_{\ell}})^{n-k}.
\end{equation}
\end{cor}

\begin{proof}
As done in~\cite{finkjmaa82}, considering the majorization relation
\begin{equation*}
(\overbrace{k,k,\dotsc,k}^{n})\prec(\overbrace{n,\dotsc,n}^k,\overbrace{0,\dotsc,0}^{n-k})
\end{equation*}
for $n>k$, the inequality~\eqref{finkjmaa82=ineq3.2} becomes
\begin{equation*}
(-1)^{nk}\bigl[f^{(k)}(t)\bigr]^n\le(-1)^{nk}\bigl[f^{(n)}(t)\bigr]^k[f(t)]^{n-k},\quad n>k>0.
\end{equation*}
Substituting $h_\ell(t)$ for $f$ in the above inequality, letting $t\to0^+$, and utilizing~\eqref{c(n)h(n)-limit} procure
\begin{gather*}
(-1)^{nk}\bigl[h_\ell^{(k)}(t)\bigr]^n\le(-1)^{nk}\bigl[h_\ell^{(n)}(t)\bigr]^k[h_\ell(t)]^{n-k},\\
(-1)^{nk}\Bigl[(-1)^k\frac{c_{\ell+k}}{\ell!}\Bigr]^n
\le(-1)^{nk}\Bigl[(-1)^n\frac{c_{\ell+n}}{\ell!}\Bigr]^k
\Bigl(\frac{c_{\ell}}{\ell!}\Bigr)^{n-k}
\end{gather*}
for $n>k>0$ and $\ell\ge0$. This may be simplified as~\eqref{cor-2nd-cauchy}. The required proof is complete.
\end{proof}

\begin{thm}\label{c(n)theorem-D}
If $\ell\ge0$, $n\ge k\ge m$, $k\ge n-k$, and $m\ge n-m$, then
\begin{equation}\label{c(n)theorem-D-eq}
c_{\ell+k}c_{\ell+n-k} \ge c_{\ell+m}c_{\ell+n-m}.
\end{equation}
\end{thm}

\begin{proof}
In~\cite[p.~397, Theorem~D]{haerc-JMAA-1996}, it was recovered that if $f(x)$ is completely monotonic on $(0,\infty)$ and if $n\ge k\ge m$, $k\ge n-k$, and $m\ge n-m$, then
\begin{equation}
(-1)^nf^{(k)}(x)f^{(n-k)}(x)\ge (-1)^nf^{(m)}(x)f^{(n-m)}(x).
\end{equation}
Replacing $f(x)$ by the function $h_\ell(t)$ in the above inequality leads to
\begin{equation*}
(-1)^nh_\ell^{(k)}(t)h_\ell^{(n-k)}(t)\ge (-1)^nh_\ell^{(m)}(t)h_\ell^{(n-m)}(t).
\end{equation*}
Further taking $t\to0^+$ and employing~\eqref{c(n)h(n)-limit} find
\begin{equation*}
(-1)^n (-1)^k\frac{c_{\ell+k}}{\ell!} (-1)^{n-k}\frac{c_{\ell+n-k}}{\ell!} \ge (-1)^n (-1)^m\frac{c_{\ell+m}}{\ell!} (-1)^{n-m}\frac{c_{\ell+n-m}}{\ell!}.
\end{equation*}
Simplifying this inequality leads to~\eqref{c(n)theorem-D-eq}. The proof of Theorem~\ref{c(n)theorem-D} is complete.
\end{proof}

\begin{thm}\label{GHI-thm}
For $n,m\in\mathbb{N}$ and $\ell\ge0$, let
\begin{align*}
\mathcal{G}_{n,m,\ell}&={c_{\ell+n+2m}}(c_{\ell})^2
-{c_{\ell+n+m}}{c_{\ell+m}}{c_{\ell}}
-{c_{\ell+n}}{c_{\ell+2m}}{c_{\ell}}
+{c_{\ell+n}}({c_{\ell+m}})^2,\\
\mathcal{H}_{n,m,\ell}&=c_{\ell+n+2m}(c_\ell)^2 -2c_{\ell+n+m}c_{\ell+m}c_\ell +c_{\ell+n}(c_{\ell+m})^2,\\
\mathcal{I}_{n,m,\ell}&=c_{\ell+n+2m}(c_\ell)^2-2c_{\ell+n}c_{\ell+2m}c_\ell +c_{\ell+n}(c_{\ell+m})^2.
\end{align*}
Then
\begin{gather}
\mathcal{G}_{n,m,\ell}\ge0,\quad \mathcal{H}_{n,m,\ell}\ge0, \\
\mathcal{H}_{n,m,\ell}\lesseqgtr \mathcal{G}_{n,m,\ell} \quad \text{when $m\lessgtr n$},
\end{gather}
and
\begin{equation}
\mathcal{I}_{n,m,\ell}\ge \mathcal{G}_{n,m,\ell}\ge0\quad \text{when $n\ge m$}.
\end{equation}
\end{thm}

\begin{proof}
In~\cite[Theorem~1 and Remark~2]{haerc-JMAA-1997}, it was obtained that if $f$ is completely monotonic on $(0,\infty)$ and
\begin{align}
G_{n,m}&=(-1)^n\bigl\{f^{(n+2m)}f^2-f^{(n+m)}f^{(m)}f-f^{(n)} f^{(2m)}f+f^{(n)}\bigl[f^{(m)}\bigr]^2\bigr\},\\
H_{n,m}&=(-1)^n\bigl\{f^{(n+2m)}f^2-2f^{(n+m)}f^{(m)}f+f^{(n)}\bigl[f^{(m)}\bigr]^2\bigr\},\\
I_{n,m}&=(-1)^n\bigl\{f^{(n+2m)}f^2-2f^{(n)}f^{(2m)}f+f^{(n)}\bigl[f^{(m)}\bigr]^2\bigr\}
\end{align}
for $n,m\in\mathbb{N}$, then $G_{n,m}\ge0$, $H_{n,m}\ge0$, and
\begin{gather}
H_{n,m}\lesseqgtr G_{n,m} \quad \text{when $m\lessgtr n$},\\
I_{n,m}\ge G_{n,m}\ge0\quad \text{when $n\ge m$}.
\end{gather}
Replacing $f(t)$ by $h_\ell(t)$ in $G_{n,m}$, $H_{n,m}$, and $I_{n,m}$ and simplifying produce
\begin{align*}
G_{n,m}&=(-1)^n\bigl\{h_\ell^{(n+2m)}h_\ell^2-h_\ell^{(n+m)}h_\ell^{(m)}h_\ell-h_\ell^{(n)} h_\ell^{(2m)}h_\ell+h_\ell^{(n)}\bigl[h_\ell^{(m)}\bigr]^2\bigr\},\\
H_{n,m}&=(-1)^n\bigl\{h_\ell^{(n+2m)}h_\ell^2-2h_\ell^{(n+m)}h_\ell^{(m)}h_\ell +h_\ell^{(n)}\bigl[h_\ell^{(m)}\bigr]^2\bigr\},\\
I_{n,m}&=(-1)^n\bigl\{h_\ell^{(n+2m)}h_\ell^2-2h_\ell^{(n)}h_\ell^{(2m)}h_\ell +h_\ell^{(n)}\bigl[h_\ell^{(m)}\bigr]^2\bigr\}.
\end{align*}
Further taking $t\to0^+$ and employing~\eqref{c(n)h(n)-limit} discover
\begin{equation*}
{(\ell!)^3}G_{n,m}=\mathcal{G}_{n,m,\ell},\quad
{(\ell!)^3}H_{n,m}=\mathcal{H}_{n,m,\ell},\quad
{(\ell!)^3}I_{n,m}=\mathcal{I}_{n,m,\ell}.
\end{equation*}
The proof of Theorem~\ref{GHI-thm} is complete.
\end{proof}

\begin{thm}\label{Norlund-add-thm}
If $m\ge1$ and $a_0,a_1,\dotsc,a_m$ be nonnegative integers, then
\begin{equation}\label{Norlund-add-1-final}
\biggl(\frac{c_{a_0}}{a_0!}\biggr)^{m-1}\frac{c_{_{\sum_{k=0}^ma_k}}}{\bigl(\sum_{k=0}^ma_k\bigr)!}\ge \prod_{k=1}^m\frac{c_{a_0+a_k}}{(a_0+a_k)!}
\end{equation}
and
\begin{equation}\label{Norlund-add-2-final}
\biggl|\frac{c_{a_i+a_j}}{(a_i+a_j)!}\biggr|_m\ge0.
\end{equation}
\end{thm}

\begin{proof}
In~\cite{Boll-Un-1991} and~\cite[p.~369 and~374]{mpf-1993}, it was obtained that if $f$ is completely monotonic on $[0,\infty)$ and $m\ge1$, then
\begin{equation}\label{Norlund-add-1}
[f(x_0)]^{m-1}f\Biggl(\sum_{k=0}^mx_k\Biggr)\ge \prod_{k=1}^mf(x_0+x_k)
\end{equation}
and
\begin{equation}\label{Norlund-add-2}
|f(x_i+x_j)|_m\ge0.
\end{equation}
We consider the function~\eqref{h(x)-dfn-eq} from an alternative viewpoint
\begin{equation}\label{h(x)-dfn-eq-s}
\mathfrak{h}(t;s)=\int_0^\infty \frac{u+1}{u[(\ln u)^2 +\pi^2]}\frac{\td u}{(u+1+t)^{s+1}}
\end{equation}
and find that $\mathfrak{h}(t;s)$ is a completely monotonic function of $s\in[0,\infty)$. Replacing the function $f$ and nonnegative numbers $x_0,x_1,\dotsc,x_m$ in~\eqref{Norlund-add-1} and~\eqref{Norlund-add-2} by the function $\mathfrak{h}(t;s)$ and nonnegative integers $a_0,a_1,\dotsc,a_m$ respectively yields
\begin{equation}\label{Norlund-add-1-0}
[\mathfrak{h}(t;a_0)]^{m-1}\mathfrak{h}_t\Biggl(\sum_{k=0}^ma_k\Biggr)\ge \prod_{k=1}^m\mathfrak{h}(t;a_0+a_k)
\end{equation}
and
\begin{equation}\label{Norlund-add-2-0}
|\mathfrak{h}(t;a_i+a_j)|_m\ge0.
\end{equation}
By virtue of~\eqref{b(n)-limit}, we obtain
\begin{equation}
\lim_{t\to0}\mathfrak{h}(t;a_i)=\frac{c_{a_i}}{a_i!}.
\end{equation}
Therefore, taking $t\to0$ in~\eqref{Norlund-add-1-0} and~\eqref{Norlund-add-2-0} leads to~\eqref{Norlund-add-1-final} and~\eqref{Norlund-add-2-final}. The proof of Theorem~\ref{Norlund-add-thm} is complete.
\end{proof}

\begin{rem}
This paper is a slightly revised and corrected version of the preprint~\cite{Norlund-No-CM.tex}.
\end{rem}

\subsection*{Acknowledgements}
The author is grateful to the anonymous referee for his/her careful corrections to and valuable comments on the original version of this paper.
\par
The author appreciates Professor Dr Yi Wang at Dalian University of Technology in China for his sincerely invitation, generously financial support, and valuably academic communication, thanks PhD Chang-Hui Hu for his warmhearted reception, and acknowledges leaders of the School of Mathematical Sciences for their kind hospitality between 3\nobreakdash--7 December 2013.
\par
The author was partially supported by the NNSF under Grant No.~11361038 of China.

\end{document}